\documentclass[12pt]{article}

\usepackage[T1]{fontenc}
\usepackage{avant}
\usepackage[cp1250]{inputenc}
\usepackage{amsmath,amssymb,amsthm}
\usepackage{amscd}
\usepackage{mathtools} 
\usepackage{stmaryrd} 
\usepackage[ruled,linesnumbered,norelsize]{algorithm2e} 
\usepackage{longtable} 
\usepackage{url} 
\usepackage{authblk} 

\usepackage{tikz}
\usepackage[all]{xy}
\newtheorem{theorem}{Theorem}

\newtheorem{definition}{Definition}

\newtheorem{corollary}{Corollary}
\newtheorem{remark}{Remark}

\title{Semisimple subalgebras in simple Lie algebras and a computational approach to the compact Clifford-Klein forms problem}

\author[1]{Maciej Boche\'nski \thanks{mabo@matman.uwm.edu.pl}}
\author[1]{Piotr Jastrz\c ebski \thanks{piojas@matman.uwm.edu.pl}}
\author[1]{Anna Szczepkowska \thanks{anna.szczepkowska@matman.uwm.edu.pl}}
\author[1]{Aleksy Tralle \thanks{tralle@matman.uwm.edu.pl}}
\author[1]{Artur Woike \thanks{awoike@matman.uwm.edu.pl}}
\affil[1]{Faculty of Mathematics and Computer Science, University of Warmia and Mazury in Olsztyn, 
 Poland }

\begin{document}
\maketitle{}
\abstract{In this  paper we develop  algorithms of finding homogeneous spaces of semisimple non-compact Lie groups which do not admit compact Clifford-Klein forms. We propose a computer program which checks if the given homogeneous space has a non-vanishing cohomological obstruction (found by Tholozan) to compact Clifford-Klein forms. By a numerical experiment we show that there is a large class of homogeneous spaces satisfying Tholozan's condition.}

\section{Introduction}\label{sec:intro}
Assume that we are given a non-compact homogeneous space $G/H$ of a reductive real Lie group $G$ and a closed subgroup $H\subset G$. Suppose that there exists a discrete subgroup $\Gamma\subset G$ which acts properly,  co-compactly and freely on $G/H$ by left translations. The quotient $\Gamma\backslash G/H$ is called a compact Clifford-Klein form. The problem of determining which reductive homogeneous spaces admit such forms goes back to Calabi and Markus \cite{CM62}  and was formulated as a research program by T. Kobayashi \cite{Kob89}. There are many results aimed at fulfilling it \cite{BEN}, \cite{BL},  \cite{BT}, \cite{BT1}, \cite{LMZ}, \cite{Ka}, \cite{Kob92}, \cite{Kob93},\cite{Kob96},\cite{Kob98}, \cite{M}, \cite{MA},\cite{OH}, \cite{Th},\cite{OW}, but in general the problem is far from being solved (see, for example  the surveys \cite{C} and \cite{KY}). Moreover, the known families of homogeneous spaces with compact Clifford-Klein forms are of a special nature \cite{C}. 
In \cite{M} and \cite{Th} Morita and Tholozan found some  topological  obstructions to the existence of compact Clifford-Klein forms. The examples obtained with the use of these works are significant, because they make advances toward {\it Kobayashi's space form conjecture}. This hypothesis states that the homogeneous space $\mathbb{H}^{p,q}=SO_0(p,q+1)/SO_0(p,q)$ has a compact Clifford-Klein form if and only if one of the following holds: $p$ is even and $q=1$, $p$ is divisible by $4$ and $q=3$, $p=8$ and $q=7$ . Using the found obstruction, Tholozan was able to prove that for  the following pairs $(G,H)$, the homogeneous spaces $G/H$ do not have compact Clifford-Klein forms (\cite{Th}, Theorem 5):
\begin{itemize}
\item $G=SO_0(p,q+r),H=SO_0(p,q);p,q>0, r>0,$ $p$ odd.
\item $G=SL(n,\mathbb{R}), H=SL(m,\mathbb{R}); 1<m<n,$ $m$ even,
\item $G=SL(p+q,\mathbb{C}), H=SU(p,q);p,q>0$,
\item $G=Sp(2(p+q),\mathbb{C}), H=Sp(p,q)$,
\item $G=SO(2n,\mathbb{C}), H=SO^*(2n)$,
\item $G=SL(p+q,\mathbb{R}),H=SO_0(p,q); p,q>1$,
\item $G=SU^*(p+q),H=Sp(p,q); p,q>1.$
\end{itemize}
On the other hand, the lack of compact Clifford-Klein forms for these homogeneous spaces was established by a different (but may be equivalent) line of reasoning by Morita \cite{M}. Also, Theorem 5 in \cite{Th} extends some previous results obtained by different methods, for example by Kobayashi and Ono \cite{KO}, Labourie, Mozes and Zimmer \cite{LMZ} and others.  
Therefore, in this work, we ask a question: {\it what are the possible families of homogeneous spaces which do not have compact Clifford-Klein forms due to Tholozan's  cohomological obstruction?} 
   The aim of the present article is to show by {\it numerical experiments} that the method of \cite{Th} can be applied to a surprisingly large variety of homogeneous spaces (Theorem \ref{thm:main}).  We  construct algorithms which detect homogeneous spaces with no compact Clifford-Klein forms because they satisfy the property found by Tholozan. We get families of homogeneous spaces $G/H$ by applying our algorithm to classes of pairs $(\mathfrak{g},\mathfrak{h})$, where $\mathfrak{h}$ are semisimple subalgebras of a simple linear Lie algebra $\mathfrak{g}$. We use the algorithms of generating such pairs developed by de Graaf \cite{G},\cite{G1} and Dietrich, Faccin and de Graaf \cite{DFG}. In order to present our results we begin by describing obstructions found in \cite{M} and \cite{Th}.
 Let $K_H\subset H$ be a maximal compact subgroup of $H$ and $K\supset K_H$ be a maximal compact subgroup of $G$ containing $K_H$. Also denote by $G_{U}$ and $H_{U}$ the compact duals of $G$ and $H$, respectively (see Section \ref{sec:prelim}  for a description of duality). 

\begin{theorem} [Morita, Theorem 1.3 in \cite{M}]
If the homomorphism
$$\pi^{\ast} : H^{\ast}(G_{U}/H_{U};\mathbb{R}) \rightarrow H^{\ast}(G_{U}/K_{H};\mathbb{R})$$
induced by the projection $\pi : G_{U}/K_{H} \rightarrow G_{U}/H_{U}$ is not injective, then G/H does not admit a compact Clifford-Klein forms.
\end{theorem}

 Let $i$ denote the canonical map $i: H_U/K_H\rightarrow G_U/K$.
\begin{theorem}[(Tholozan,\cite{Th}]\label{thm:omega} Let $G/H$ be a reductive homogeneous space, with $G$ and $H$ connected and with finite center. Set $p=\dim G/H-\dim K/K_H$. Then there exists a $G$-invariant $p$-form $\omega_{G,H}$ on $G/K$ such that, for any torsion-free discrete subgroup $\Gamma\subset G$ acting properly discontinuosly on $G/H$ we have
$$vol(\Gamma\setminus G/H)=|\int_{[\Gamma]}\omega_{G,H}|.$$
\end{theorem}
Thus, $\omega_{G,H}$ yields an obstruction in some cases: if it is zero, no $\Gamma$ can yield a compact Clifford-Klein form. In particular, Tholozan found some conditions on $G/H$ which ensure $\omega_{G,H}=0$. Also,  the dual homogeneous spaces $G/K$ and $G_U/K$ have the following property. By construction, the tangent spaces at the base point $x_0=K$ in $G/K$ and $G_U/K$ are isomorphic as representations of $K$. This induces an isomorphism between the exterior algebras of invariant forms on $G/K$ and $G_U/K$. If $\alpha$ is a $G$-invariant form on $G/K$, the image of $\alpha$ by this isomorphism will be denoted by $\alpha^U$. 
\begin{theorem}[Tholozan, \cite{Th}, Theorem 3.2]\label{thm:vanishing-omega} The cohomology class of the form
$${1\over vol(G_U/H_U)}\omega^{U}_{G,H}\in H^*(G_U/K,\mathbb{R})$$
is Poincar\'e dual to the homology class $i_*[H_U/K_H]$.
\end{theorem}

The main application of the algorithms developed in this article is formulated as follows.
\begin{theorem}\label{thm:main} Let $G/H$ be a homogeneous space of non-compact simple linear Lie group $G$ of exceptional type.  If the pair  $(\mathfrak{g},\mathfrak{h})$ is contained in one of the tables in Section \ref{sec:tables}, then $G/H$  does not have compact Clifford-Klein forms.
\end{theorem}
Note that our algorithm does not find  all $G/H$ with no compact Clifford-Klein forms. Indeed, there exist homogeneous spaces for which  Tholozan's obstruction vanishes, but still they have no Clifford-Klein forms. Also, we use \cite{DFG} in the simplest possible cases of regular subalgebras in $E_{6(6)}$ and split Lie algebras $\mathfrak{g}$. This already yields large families. On the other hand, generating pairs $(\mathfrak{g},\mathfrak{h})$ by the algorithms \cite{DFG} clearly may yield much more examples. Note that we choose exceptional Lie groups just for simplicity, because the algorithms are applicable to the classical case as well. 

 Finally let us mention an easy but interesting result formulated below, which is a kind of byproduct of our analysis and does not use a computational approach.  
\begin{theorem}\label{thm:sl2}
Let $G$ be a semisimple connected linear Lie group. Assume that the maximal compact subgroup $K$ is semisimple. For any  monomorphism $SL(2,\mathbb{R})\hookrightarrow G$ the homogeneous space $G/SL(2,\mathbb{R})$ does not have compact Clifford-Klein forms.
\label{c2}
\end{theorem}
 
   To put this observation into the general context, recall the following.
\begin{itemize}
	\item The space $SL(3,\mathbb{R})/SL(2,\mathbb{R})$ does not have compact Clifford-Klein forms for any embedding $SL(2,\mathbb{R})\subset SL(3,\mathbb{R})$ (see \cite{BEN}).
	\item The space $SL(m,\mathbb{R})/SL(2,\mathbb{R}),$ $m\geq 5$ does not have compact Clifford-Klein forms for a standard embedding $SL(2,\mathbb{R})\subset SL(m,\mathbb{R})$ (see \cite{LZ}).
	\item The space $SL(4,\mathbb{R})/SL(2,\mathbb{R}),$  does not have compact Clifford-Klein forms for a standard embedding $SL(2,\mathbb{R})\subset SL(4,\mathbb{R})$ (see \cite{SH}).
	\item The space $SL(m,\mathbb{R})/f(SL(2,\mathbb{R})),$ $m\geq 4$ does not have compact Clifford-Klein forms where $f:SL(2,\mathbb{R})\rightarrow SL(m,\mathbb{R})$ is a real, n-dimensional irreducible representation (see \cite{MA} and \cite{OH}).
\end{itemize} 
\noindent {\bf Acknowledgement}. We thank Willem de Graaf for his help in getting some literature sources.
\section{Computational setup}\label{sec:comp}
We have implemented the algorithms in the language of the computer algebra system GAP4.8 \cite{GAP} using the package NoCK \cite{BJSTW}. We create pairs (Lie algebra, semisimple subalgebra) using the database CoReLG \cite{DFG1} and SLA \cite{G}. 
\section{Preliminaries}\label{sec:prelim}
 Throughout this paper we use the basics of Lie theory without further explanations. One can consult \cite{O1}. We denote Lie groups by $G, H,...$, and their Lie algebras by the corresponding Gothic letters $\mathfrak{g},\mathfrak{h}...$.  The symbol $\mathfrak{g}^{\mathbb{C}}$ denotes the complexification of a real Lie algebra $\mathfrak{g}$. Our notation is in accordance with \cite{O1}.

In this article we use various results dealing with  homogeneous spaces $G/H$ of a {\it reductive type}. Recall this notion following \cite{Kob89}. By a real reductive linear Lie group $G$ we mean a real linear Lie group contained in a connected complex reductive Lie group $G^{\mathbb{C}}$ whose Lie algebra $\mathfrak{g}^{\mathbb{C}}$ is isomorphic to $\mathfrak{g}\otimes_{\mathbb{R}}\mathbb{C}$. Without further explanations we use  the notions of the {\it Cartan involution} $\theta$ and the {\it Cartan decomposition} $\mathfrak{g}=\mathfrak{k}+\mathfrak{p}$.  Let $H$ be a closed subgroup of a real reductive linear group $G$ and by $K$ denote a maximal compact subgroup in $G$. We say that $H$ is reductive in $G$ and $G/H$ is of a {\it reductive type} if there exists a Cartan involution $\theta$ of $G$ such that $G$ has a polar decomposition $H=(H\cap K)\cdot\exp(\mathfrak{h}\cap\mathfrak{p})$ and connected Lie subgroup corresponding to $\mathfrak{h}^{\mathbb{C}}$ is closed in $\operatorname{Int}(\mathfrak{g}^{\mathbb{C}})$. Note that semisimple linear real Lie groups are reductive. Our basic assumption in this article is that $G$ is simple and linear.

Using the Cartan involution, one can write the {\it compatible Cartan decomposition}
  $$\mathfrak{h}=\mathfrak{k}_H+\mathfrak{p}_H, \ \mathfrak{k}_H=\mathfrak{k}\cap\mathfrak{h}, \ \mathfrak{p}_H=\mathfrak{p}\cap\mathfrak{h},$$
  on the Lie algebra level. Throughout this article we write 
  $$d(G):=\dim\mathfrak{p},\,d(H):=\dim\mathfrak{p}_H.$$
 These numbers are often called {\it non-compact dimensions} of $G$ and $H$.
 
   Let $\mathfrak{b}\subset\mathfrak{p}$ be a maximal abelian subspace in $\mathfrak{p}$. All such subspaces are conjugate by some automorphism of the form $\operatorname{Ad}\,k$ for some $k\in K$. The dimension $\dim\mathfrak{b}$ is called the real rank of $\mathfrak{g}$. The real rank of a real Lie algebra $\mathfrak{g}$ will be denoted by $\operatorname{rank}_{\mathbb{R}}\mathfrak{g}$.  
   
   For simplicity of the implementation of the algorithms we consider examples of {\it split} semisimple Lie algebras (see \cite{DFG} and \cite{O1}). These are defined as real forms of complex semisimple Lie algebras $\mathfrak{g}^{\mathbb{C}}$ generated by the Chevalley bases of $\mathfrak{g}^{\mathbb{C}}$ over the reals.
  
  Let $X$ be a Hausdorff topological space and $\Gamma$ a topological group acting on $X$. We say that an action of $\Gamma$ on $X$ is {\it proper} if for any compact subset $S\subset X$ the set
  $$\{\gamma\in\Gamma\,|\,\gamma(S)\cap S\not=\emptyset\}$$
  is compact. In particular, this article is devoted to the proper actions of a discrete subgroup $\Gamma\subset G$ on $G/H$ by left translations, and for simplicity we assume that the action is free. 
  Since we  consider homogeneous spaces $G/H$ of a reductive type  and assume that $G$ is connected and (semi)simple, we know that every $G/H$ has a dual $G_U/H_U$, which we now describe (see \cite{KO}, Section 3).  
   Let $G_{U}$ be a compact real form of a (connected) complexification $G^{\mathbb{C}}$ of $G$ and let $H_{U}$ be a compact real form of $H^{\mathbb{C}}\subset G^{\mathbb{C}}$. The space $G_{U}/H_{U}$ is called the  homogeneous space of compact type associated with $G/H$ (or dual to $G/H$).
   
   Our algorithms use also the well known Calabi-Markus phenomenon (see Theorem \ref{thm:c-m}).
   \begin{theorem}[\cite{Kob89}]\label{thm:c-m} Let $G/H$ be a non-compact homogeneous space of a reductive type and $\operatorname{rank}_{\mathbb{R}}\mathfrak{g}=\operatorname{rank}_{\mathbb{R}}\mathfrak{h}$, then $G/H$ does not have compact Clifford-Klein forms.
   \end{theorem} 
   
   In this paper we consider the homology and cohomology over the reals. If $f:X\rightarrow Y$ is a continuous map, we write $f_*: H_*(X,\mathbb{R})\rightarrow H_*(Y,\mathbb{R})$ for the map induced in homology and $f^*: H^*(Y,\mathbb{R})\rightarrow H^*(X,\mathbb{R})$ for the map induced in cohomology. If $X$ is a manifold of dimension $n$ we denote by $[X]\in H_n(X,\mathbb{R})$ its fundamental class. For any topological space $X$ we denote by $P(X,t)$ the Poincar\'e polynomial of $X$. Note that all topological notions and facts we need can be found in \cite{FHT}.
\section{On generating semisimple subalgebras of real simple Lie algebras}\label{sec:subalgebras}

Let $\mathfrak{g}^{\mathbb{C}}$ be a simple complex Lie algebra and $\mathfrak{a}^{\mathbb{C}}$ be a Cartan subalgebra. We say that a subalgebra $\mathfrak{h}^{\mathbb{C}}$ is regular with respect to $\mathfrak{a}^{\mathbb{C}}$, if $[\mathfrak{a}^{\mathbb{C}},\mathfrak{h}^{\mathbb{C}}]\subset\mathfrak{h}^{\mathbb{C}}$. It is well 
known  that all Cartan subalgebras in $\mathfrak{g}^{\mathbb{C}}$ are conjugate, therefore, there is no need of specifying them. Let $\mathfrak{g}$ be a real non-compact simple Lie algebra with a Cartan decomposition $\mathfrak{g}=\mathfrak{k}+\mathfrak{p}$ and associated Cartan involution $\theta$. Let $G$ be an adjoint group of $\mathfrak{g}$. It is known that every Cartan subalgebra in $\mathfrak{g}$ is $G$-
conjugate to a $\theta$-stable Cartan subalgebra. By well known classical results of Sugiura and Kostant, up to $G$-conjugacy, there are a finite number of $\theta$-stable Cartan subalgebras. A subalgebra $\mathfrak{h}$ of $\mathfrak{g}$ is called regular with respect to a Cartan subalgebra $\mathfrak{a}$, if $[\mathfrak{a}^{\mathbb{C}},\mathfrak{h}^{\mathbb{C}}]\subset\mathfrak{h}^{\mathbb{C}}$. Dynkin \cite{D} developed an algorithm to classify the 
regular semisimple subalgebras of $\mathfrak{g}^{\mathbb{C}}$  up to a conjugacy by the inner automorphism group. Dietrich, Faccin and de Graaf \cite{DFG} created an algorithm to classify the regular semisimple subalgebras of a real simple Lie algebra $\mathfrak{g}$ (with respect to any $\theta$-stable Cartan subalgebra).  This algorithm has been implemented in the software package CoRelG \cite{DFG1} for the computer algebra system GAP \cite{GAP}. We use these results in implementing our algorithm of computing Tholozan's obstruction. We create regular subalgebras $\mathfrak{h}$ according to \cite{DFG} and apply our algorithm.      
 
\section{Formality}
In this work we use a homotopic property called formality. It is an important  property of topological spaces. For instance, the rational homotopy type of a formal topological space is determined by its cohomology algebra over $\mathbb{Q}$ \cite{FHT}. 
\subsection{Formality property} Explaining the material of this section we follow \cite{FHT}. Here $\mathbb{K}$ denotes any field of zero characteristic. We consider the category of commutative graded differential algebras (or, in the terminology of \cite{FHT}, cochain algebras). If $(A,d)$ is a differential graded algebra with a grading $S=\oplus_pA^p$, the degree $p$ of $a\in A^p$ is denoted by $|a|$. Given a vector space $V$, consider the algebra
$$\Lambda V=S(V^{\text{even}})\otimes \Lambda(V^{\text{odd}});$$
that is, $\Lambda V$ denotes a free algebra which is a tensor product of a symmetric algebra over the vector space $V^{\text{even}}$ of elements of even degrees, and an exterior algebra over the vector space $V^{\text{odd}}$ of elements of odd degrees. 
We will use the following notation:
\begin{itemize}
\item by $\Lambda V^{\leq p}$ and $\Lambda V^{>p}$ we denote the subalgebras generated by elements of degree $\leq p$ and of degree $>p$, respectively;
\item if $v\in V$ is a generator, $\Lambda v$ denotes the subalgebra generated by $v\in V$;
\item $\Lambda^pV=\langle v_1\cdots v_p\rangle,\Lambda^{\geq q}V=\oplus_{i\geq q}\Lambda^iV,\,\Lambda^+V=\Lambda^{\geq 1}V.$
\end{itemize}
\begin{definition} A {\it Sullivan algebra} is a commutative graded differential algebra of the form $(\Lambda V,d)$, where
\begin{itemize}
\item $V=\oplus_{p\geq 1}V^p$,
\item $V$ admits an increasing filtration
$$V(0)\subset V(1)\subset \cdots\subset V=\cup_{k=0}^\infty V(k)$$
with the property $d=0$ on $V(0)$, $d: V(k)\rightarrow \Lambda V(k-1),\,k\geq 1$.
\end{itemize}
\end{definition}
\begin{definition} A Sullivan algebra $(\Lambda V,d)$ is called {\it minimal} if
$$\operatorname{Im}\,d\subset\Lambda^+V\cdot\Lambda^+V.$$
\end{definition}
\begin{definition} A {\it Sullivan model} of a commutative graded differential algebra $(A,d_A)$ is a morphism 
$$m: (\Lambda V,d)\rightarrow (A,d_A)$$
inducing an isomorphism $m^*: H^*(\Lambda V,d)\rightarrow H^*(A,d_A)$.
\end{definition}
If $X$ is a $CW$-complex, there is a cochain algebra $(A_{PL}(X),d)$ of polynomial differential forms. For a smooth manifold $X$ we take a smooth triangulation of $X$ and as the model of $X$ we take the Sullivan model of $A_{PL}(X)$. If it is minimal, it is called the {\it Sullivan minimal model} of $X$. It is known \cite{FHT} that any $CW$-complex admits a Sullivan minimal model, say $(\mathcal{M}_X,d)$, which is unique up to isomorphism.
\begin{definition} Two graded differential algebras $(A,d_A)$ and $(B,d_B)$ are called {\it quasi-isomorphic} if there is a chain of morphisms 
$$(A,d_A)\rightarrow (A_1,d_1)\rightarrow \cdots\leftarrow (A_i,d_i)\rightarrow \cdots\leftarrow (B,d_B)$$
inducing isomorphisms in cohomology.
\end{definition}
\begin{definition} We say that a $CW$-complex is {\it formal} if the Sullivan minimal model $(\mathcal{M}_X,d)$ is quasi-isomorphic to the graded differential algebra $(H^*(\mathcal{M}_X,0))$. 
\end{definition}   
Note that not all $CW$ complexes are formal. For example, non-vanishing Massey products are an obstruction to formality \cite{TO}. 

On the other hand, in our arguments we will need the following.
\begin{theorem}[\cite{TO}, Example 3.3]\label{thm:symm-formal} Compact Riemannian symmetric spaces are formal.
\end{theorem}

\subsection{Cohomology and formality of compact homogeneous spaces}
In this subsection we recall the classical results on the real cohomology and formality of compact homogeneous spaces in the suitable  form \cite{T}. Note that here (and only here) we assume that  $G$ is a {\it compact} connected Lie group and $H$ is a closed subgroup. Since we consider only cohomology and homology with real coefficients we omit the field of coefficents in the notation. The cohomology algebra $H^*(G)$ is an exterior algebra over vector space $P_G$ generated by primitive elements $y_1,...,y_l$ of odd degrees $2p_1-1,...,2p_l-1$, where $l=\operatorname{rank}\,G$. These generators can be chosen in such a way, that the cohomology algebra of the classifying space $BG$ is isomorphic to the polynomial algebra $\mathbb{R}[x_1,...,x_l]$, where elements $y_j$ are cohomology classes in the universal space $EG$ containing coboundaries of the transgression for $x_j$.  Choose maximal tori $T_H\subset T$ of $H$ and $G$ respectively and their Lie algebras $\mathfrak{t}_H\subset\mathfrak{t}$. It is well known that the cohomology algebras $H^*(BG)$ and $H^*(BH)$ are isomorphic to the polynomial algebras of the invariants of the Weyl groups $W(G)$ and $W(H)$ respectively. Denote these rings as $\mathbb{R}[\mathfrak{t}]^{W(G)}$ and $\mathbb{R}[\mathfrak{t}_H]^{W(H)}$. It is a classical result that the canonical homomorphism $\rho^*(H,G): H^*(BG)\rightarrow H^*(BH)$ induced by the inclusion $H\hookrightarrow G$ sends any invariant polynomial from $H^*(BG)$ to its restriction on $H^*(BH)$. 
Introduce the following notation:
$$H^*(BG)\cong\mathbb{R}[\mathfrak{t}]^{W(G)}=\mathbb{R}[f_1,...,f_l],$$
$$H^*(BH)\cong \mathbb{R}[\mathfrak{t}_H]^{W(H)}=\mathbb{R}[u_1,...,u_m],$$
where $f_j$ are $W(G)$-invariant polynomials over $\mathfrak{t}$, and $u_k$ are $W(H)$-invariant polynomials over $\mathfrak{t}_H$, generating the rings of invariants $\mathbb{R}[\mathfrak{t}]^{W(G)}$ and $\mathbb{R}[\mathfrak{t}_H]^{W(H)}$, respectively.
The cohomology algebra of $G/H$ can be calculated with the use of the following.
\begin{theorem}[\cite{T}, Theorem 8]\label{thm:g/h} Assume that $f_1,...,f_l$ are generators of the ring $\mathbb{R}[\mathfrak{t}]^{W(G)}$ and assume that $m=\operatorname{rank}\,H$.  Let $j^*:\mathbb{R}[\mathfrak{t}]^{W(G)}\rightarrow \mathbb{R}[\mathfrak{t}_H]^{W(H)}$ denote the restriction map $j^*(f)=f|_{\mathfrak{t}_H}$ (which corresponds to $\rho^*(G,H)$ under the suitable isomorphisms). If $j^*(f_{m+1}),...,j^*(f_l)$ belong to the ideal $(j^*(f_1),...,j^*(f_m))$ generated by polynomials \newline$j^*(f_1),...,j^*(f_m)$, then
\begin{itemize}
\item $H^*(G/H)=(\mathbb{R}[\mathfrak{t}_H]^{W(H)}/(j^*(f_1),...,j^*(f_m))\otimes\Lambda (x_{m+1},...,x_l),$
\item $G/H$ is formal.
\end{itemize}
\end{theorem}
As usual, the differential graded algebra
$$(C_{G/H},d)=(H^*(BH)\otimes \Lambda (y_1,...,y_l),d)$$
$$d|_{H^*(BH)}=0,\,dy_j=j^*(f_j),\,j=1,...,l$$
is called the {\it Cartan algebra} of $G/H$. Note that $H^*(G/H)\cong H^*(C_{G/H})$ for any homogeneous space $G/H$ of a compact Lie group $G$.
Consider the exterior algebras
$$H^*(G)\cong\Lambda\,P_G\cong\Lambda (y_1,...,y_l),$$$$H^*(H)\cong\Lambda\,P_H\cong\Lambda (x_1,...,x_m)$$

\begin{theorem}[\cite{O}, p.211]\label{thm:poincare} If $G/H$ is formal, then
$$P(G/H,t)=\prod_{j=m+1}^l(1+t^{2p_j-1})\prod_{i=1}^m(1-t^{2p_i})/(1-t^{2q_i}).$$
\end{theorem}
Note that even if the gradings of $P_G$ and $P_H$ are known, in general it is not possible to write down the Poincar\'e polynomial of $G/H$ without the knowledge of the embedding of $H$ into $G$, because it influences the order of the generators. As usual, we call $p_i$ and $q_j$ the exponents of $G$ and $H$, respectively.

\section{Tholozan's obstruction to compact Clifford-Klein forms in an algorithmic fashion }
\begin{theorem}\label{thm:d} let $G/H$ be a reductive homogeneous space of a connected non-compact semisimple real linear Lie group $G$. Assume that $d=\dim\,H_U/K_H=d(H)$. If the coefficient in degree $d$ of the  Poincar\'e polynomial $P(G_U/K,t)$ vanishes, then $G/H$ does not admit compact Clifford-Klein forms.
\end{theorem}
\begin{proof} Recall that if $G/H$ admitted a compact Cliffod-Klein form, one would have $i_*[H_U/K_H]\not=0$, by Theorem \ref{thm:omega} and Theorem \ref{thm:vanishing-omega} applied together. Since $d=\dim (H_U/K_H)$, $H_d(G_U/K)$ must be non-zero, as well as $H^d(G_U/K)$. On the other hand, by the assumption of the theorem, $H^d(G_U/K)=0$, a  contradiction.
\end{proof}
\begin{corollary}\label{cor:algor-base}  Let
$$\Lambda\,P_{G_U}=\Lambda (y_1,...,y_l),\,\Lambda\, P_K=\Lambda (x_1,...,x_m)$$
be the exterior algebras over the spaces $P_G$ and $P_H$ of the primitive elements in $H^*(G)$ and $H^*(K)$ respectively. Denote the degrees as follows
$$|y_j|=2p_j-1,j=1,...,l,$$$$|x_i|=2q_i-1,\,i=1,...,m.$$
 Consider the polynomial
$$P(t)=\prod_{j=m+1}^l(1+t^{2p_j-1})\prod_{i=1}^m(1-t^{2p_i})/(1-t^{2q_i}). \eqno (1)$$
If the coefficient in degree $d=d(H)$ of $P(t)$ vanishes, then $G/H$ does not admit a compact Clifford-Klein form.
\end{corollary}
\begin{proof} Since $G_U/K$ is a symmetric space, it is formal by Theorem \ref{thm:symm-formal}. By Theorem \ref{thm:poincare} the Poincar\'e polynomial of $G/K$ has the form $(1)$. The proof follows from Theorem \ref{thm:d}.
\end{proof}

 \section{Algorithms detecting homogeneous spaces with no Clifford-Klein forms}
 Our algorithms are based on Corollary \ref{cor:algor-base}. Given a homogeneous space $G/H$ we need to do the following:
 \begin{itemize}
 \item eliminate homogeneous spaces $G/H$ which do not have compact Clifford-Klein forms by the Calabi-Markus phenomenon (\ref{thm:c-m});
 \item write down the exponents $p_1,...,p_l$ of $G$ and $q_1,...,q_m$ of $K$;
 \item calculate $d=d(H)$;
 \item write down the polynomial $(1)$ from Corollary \ref{cor:algor-base};
 \item write down all possible polynomials which are obtained from $(1)$ by permutation of the set $P=\{p_1,...,p_l\}$ and delete all obtained expressions which are rational functions but not polynomials (these do not correspond to any formal homogeneous space);
 \item if at least one of obtained polynomials has non-zero coefficient in degree $d$, exclude such pair $(G,H)$ from the consideration;
 \item if all permutations yield polynomials with zero coefficient in degree $d$, write down the output $G/H$ as a homogeneous space with no compact Clifford-Klein forms.
 \end{itemize}
 \begin{remark} {\rm Note that our algorithm may not detect some homogeneous spaces $G/H$ which do not have compact Clifford-Klein forms by Tholozan's theorem. This may happen, because of the permutation of $P=\{p_1,...,p_l\}$. We need to do this operation, because we don't know the embedding  $\mathfrak{h}\hookrightarrow\mathfrak{g}$, and, therefore, the order of $p_i$ in the expression $(1)$. However, this cannot happen if $\operatorname{rank}\,G=\operatorname{rank}\,K$. Basically, this condition is satisfied for almost all cases. Here is the list of exceptions ($\operatorname{rank}\,G>\operatorname{rank}\,K$):}
 $$G=SL(n,\mathbb{R}),\,G=SU^*(n),\,G=SO(2k+1,2m+1),$$$$G=E_{6(6)},\,G=E_{6(-26)}.$$
 \end{remark}

  In order to get classes of homogeneous spaces $G/H$ without compact Clifford-Klein forms, we use the method of computer generation of (all)  semisimple subalgebras in exceptional Lie algebras developed in \cite{DFG}.
  
\subsection{Algorithm  1}

\begin{algorithm}[H]
  \caption{\tt RegularNoCK($\mathfrak{g}$)}
  \footnotesize
  \label{alg1}
  \tcc{
  $\mathfrak{g}$ - noncompact real simple Lie algebra. 
   Return list $L$ of regular subalgebras $\mathfrak{h}$ of $\mathfrak{g}$ that corresponding homogeneous space $G/H$ do not have compact Clifford-Klein form.}
  \Begin{
  let $L0$ be a set of all regular subalgebras of $\mathfrak{g}$ (use Algoritm from \S 8, \cite{DFG}) \;
  set $L:=\emptyset$, $L1:=\emptyset$, $L2:=\emptyset$, $D:=\emptyset$, $D2:=\emptyset$\;
  \ForAll{$\mathfrak{h} \in L0$}{
    \tcp{Calabi-Markus phenomenon}
    \lIf{$\mathrm{rank}_{\mathbb{R}} \mathfrak{g}>\mathrm{rank}_{\mathbb{R}} \mathfrak{h}$}{add $\mathfrak{h}$ to $L1$} 
    }
  \ForAll{$\mathfrak{h} \in L1$}{
  compute $d=d(H)$ for corresponding $H$ for $\mathfrak{h}$\;
  add tuple $[ \mathfrak{h}, d ] $ to $L2$ \;
  \lIf{$d \notin D$}{add $d$ to $D$} 
    }
  let  $\{p_1, \ldots, p_l\}$ - the exponents for $G_U$ and $\{q_1,\ldots, q_m\}$ exponents for $K$ (as in Corollary \ref{cor:algor-base})\;
  \ForAll{$d \in D$}{
    set $\mathtt{isOk}:=\mathtt{true}$\;
    \ForAll{$\sigma \in S_l$ ($S_l$ - set of permutations of set $\{1,\ldots, l\}$)}{
      compute $P(t)$ as in Corollary \ref{cor:algor-base}  for $\{p_{\sigma(1)}, \ldots, p_{\sigma(l)}\}$ and $\{q_1,\ldots, q_m\}$\;
          \If{$P(t)$ is polynomial}{
            \If{coefficient in degree $d$ of $P(t)$ does not vanish}{$\mathtt{isOk}:=\mathtt{false}$; break;}
          }
      }
      \lIf{$\mathtt{isOk}$}{add $d$ to $D2$}
    }
\ForAll{$[\mathfrak{h},d] \in L2$}{
\lIf{$d \in D2$}{ add $\mathfrak{h}$ to $L$}
}
  \Return{$L$}\;
  
  }
  \end{algorithm}

\subsection{Algorithm 2 - split version}

\begin{algorithm}[H]
  \caption{\tt SplitNoCK($\mathfrak{g}$)}
  \footnotesize
  \label{alg2}
  \tcc{
  $\mathfrak{g}$ - noncompact real split exceptional simple Lie algebra. 
   Return list $L$ of split regular subalgebras $\mathfrak{h}$ of $\mathfrak{g}$ that corresponding homogeneous space $G/H$ do not have compact Clifford-Klein form.}
  \Begin{
  let $L0$ be a set of all regular subalgebras of $\mathfrak{g}^{\mathbb{C}}$ (use Table 11 in \cite{D} or SLA plugin for GAP-software, \cite{G}) \;
  set $L:=\emptyset$, $L1:=\emptyset$, $L2:=\emptyset$, $D:=\emptyset$, $D2:=\emptyset$\;
  \ForAll{$\mathfrak{h}^{\mathbb{C}} \in L0$}{
    set $\mathfrak{h}$ as split real form of $\mathfrak{h}^{\mathbb{C}}$\;
    \tcp{Calabi-Markus phenomenon}
    \lIf{$\mathrm{rank}_{\mathbb{R}} \mathfrak{g}>\mathrm{rank}_{\mathbb{R}} \mathfrak{h}$}{add $\mathfrak{h}$ to $L1$} 
    }
  \ForAll{$\mathfrak{h} \in L1$}{
  compute $d=d(H)$ for corresponding $H$ for $\mathfrak{h}$\;
  add tuple $[ \mathfrak{h}, d ] $ to $L2$ \;
  \lIf{$d \notin D$}{add $d$ to $D$} 
    }
  let  $\{p_1, \ldots, p_l\}$ - the exponents for $G_U$ and $\{q_1,\ldots, q_m\}$ exponents for $K$ (as in Corollary \ref{cor:algor-base})\;
  \ForAll{$d \in D$}{
    set $\mathtt{isOk:=true}$\;
    \ForAll{$\sigma \in S_l$ ($S_l$ - set of permutations of set $\{1,\ldots, l\}$)}{
      compute $P(t)$ as in Corollary \ref{cor:algor-base}  for $\{p_{\sigma(1)}, \ldots, p_{\sigma(l)}\}$ and $\{q_1,\ldots, q_m\}$\;
          \If{$P(t)$ is polynomial}{
            \If{coefficient in degree $d$ of $P(t)$ does not vanish}{$\mathtt{isOk:=false}$; break;}
          }
      }
      \lIf{$\mathtt{isOk}$}{add $d$ to $D2$}
    }
\ForAll{$[\mathfrak{h},d] \in L2$}{
\lIf{$d \in D2$}{ add $\mathfrak{h}$ to $L$}
}
  \Return{$L$}\;
  
  }
\end{algorithm}

  \section{Proof of Theorem \ref{thm:main}}
  Using  our algorithms and the computer program we create tables in Section \ref{sec:tables}, and Corollary \ref{cor:algor-base} guarantees that all such homogeneous spaces do not admit Clifford-Klein forms. In greater detail, we obtain two tables according to the following procedures. 
  \begin{itemize}
  \item Table 1: we take the list of all semisimple regular subalgebras $\mathfrak{h}$ in $\mathfrak{g}$ of type $\mathfrak{e}_{6(6)}$ obtained in \cite{DFG}, and apply Algorithm 1 from subsection 6.1. 
  \item Table 2: here we restrict ourselves to only split pairs $(\mathfrak{g},\mathfrak{h})$. In this case, $\mathfrak{h}\subset\mathfrak{g}$ if and only if $\mathfrak{h}^c\subset\mathfrak{g}^c$. It follows that we can use either (corrected) Dynkin's classification of semisimple subalgebras of complex simple Lie algebras \cite{D}, or \cite{G1}. We apply Algorithm 2 from subsection 6.2 and get the resulting table.
  \end{itemize}
   
  \section{Tables}\label{sec:tables}
  Here we illustrate our algorithms by lists of new homogeneous spaces of exceptional Lie groups which do not have compact Clifford-Klein forms. 
  
\begin{longtable}{|c|c|}
\kill
\caption{List of pairs $(\mathfrak{g}=\mathfrak{e}_{6(6)},\mathfrak{h}$) that corresponding homogeneous spaces $G/H$ do not have compact Clifford-Klein forms (Algorithm \ref{alg1}).}\label{tab:reg}\\
\hline $\mathfrak{g}$ & $\mathfrak{h}$ \\\hline
 $\mathfrak{e}_{6(6)}$ & $   \mathfrak{sl}(2,\mathbb{R}) $ \\\hline 
 $\mathfrak{e}_{6(6)}$ & $   \mathfrak{sl}(3,\mathbb{R}) $ \\\hline 
 $\mathfrak{e}_{6(6)}$ & $   \mathfrak{sl}(5,\mathbb{R}) $ \\\hline 
 $\mathfrak{e}_{6(6)}$ & $   \mathfrak{sl}(6,\mathbb{R}) $ \\\hline 
 $\mathfrak{e}_{6(6)}$ & $   \mathfrak{sl}(2,\mathbb{R}) \oplus\mathfrak{sl}(2,\mathbb{R}) $ \\\hline 
 $\mathfrak{e}_{6(6)}$ & $   \mathfrak{sl}(3,\mathbb{R}) \oplus\mathfrak{sl}(2,\mathbb{R}) $ \\\hline 
 $\mathfrak{e}_{6(6)}$ & $   \mathfrak{sl}(3,\mathbb{R}) \oplus\mathfrak{sl}(3,\mathbb{R}) $ \\\hline 
 $\mathfrak{e}_{6(6)}$ & $   \mathfrak{sl}(4,\mathbb{R}) \oplus\mathfrak{sl}(2,\mathbb{R}) $ \\\hline 
 $\mathfrak{e}_{6(6)}$ & $   \mathfrak{sl}(2,\mathbb{R}) \oplus\mathfrak{sl}(2,\mathbb{R}) \oplus\mathfrak{sl}(2,\mathbb{R}) $ \\\hline 
 $\mathfrak{e}_{6(6)}$ & $   \mathfrak{sl}(3,\mathbb{R}) \oplus\mathfrak{sl}(3,\mathbb{R}) \oplus\mathfrak{sl}(2,\mathbb{R}) $ \\\hline 
 $\mathfrak{e}_{6(6)}$ & $   \mathfrak{sl}(4,\mathbb{R}) \oplus\mathfrak{sl}(2,\mathbb{R}) \oplus\mathfrak{sl}(2,\mathbb{R}) $ \\\hline 
  $\mathfrak{e}_{6(6)}$ & $   \mathfrak{sl}(2,\mathbb{C})  $ \\\hline 
   $\mathfrak{e}_{6(6)}$ & $   \mathfrak{sl}(2,\mathbb{R})\oplus   \mathfrak{sl}(2,\mathbb{C}) $ \\\hline 
 
      $\mathfrak{e}_{6(6)}$ & $   \mathfrak{so}(3,5)  $ \\\hline 
           $\mathfrak{e}_{6(6)}$ & $\mathfrak{sl}(2,\mathbb{C})\oplus   \mathfrak{su}(2,2)  $ \\\hline 
              $\mathfrak{e}_{6(6)}$ & $   \mathfrak{sl}(2,\mathbb{R})\oplus \mathfrak{sl}(2,\mathbb{R})\oplus   \mathfrak{sl}(2,\mathbb{C}) $ \\\hline 
$\mathfrak{e}_{6(6)}$ & $   \mathfrak{sl}(2,\mathbb{C})\oplus \mathfrak{sl}(2,\mathbb{C}) $ \\\hline
$\mathfrak{e}_{6(6)}$ & $   \mathfrak{sl}(2,\mathbb{C})\oplus \mathfrak{su}(1,2) $ \\\hline
$\mathfrak{e}_{6(6)}$ & $   \mathfrak{sl}(2,\mathbb{R})\oplus \mathfrak{sl}(3,\mathbb{C}) $ \\\hline 
$\mathfrak{e}_{6(6)}$ & $   \mathfrak{su}(1,2) $ \\\hline 
$\mathfrak{e}_{6(6)}$ & $ \mathfrak{sl}(3,\mathbb{C}) \oplus \mathfrak{su}(1,2) $ \\\hline 
$\mathfrak{e}_{6(6)}$ & $ \mathfrak{su}(2)\oplus\mathfrak{su}^{\star}(4) $ \\\hline 
$\mathfrak{e}_{6(6)}$ & $ \mathfrak{su}(2) \oplus\mathfrak{sl}(2,\mathbb{C})  $ \\\hline
$\mathfrak{e}_{6(6)}$ & $ \mathfrak{su}^{\star}(4) $ \\\hline 
$\mathfrak{e}_{6(6)}$ & $ \mathfrak{su}^{\star}(6) $ \\\hline 
$\mathfrak{e}_{6(6)}$ & $ \mathfrak{su}(2)\oplus\mathfrak{su}^{\star}(6) $ \\\hline  
$\mathfrak{e}_{6(6)}$ & $ \mathfrak{su}(2)\oplus\mathfrak{su}(2)\oplus\mathfrak{su}^{\star}(4) $ \\\hline 
$\mathfrak{e}_{6(6)}$ & $ \mathfrak{su}(2)\oplus\mathfrak{su}(2)\oplus\mathfrak{sl}(2,\mathbb{C}) $ \\\hline
\end{longtable}

\begin{longtable}{|c|c|}
\kill
\caption{List of pairs $(\mathfrak{g},\mathfrak{h})$ (both split) that corresponding homogeneous spaces $G/H$ do not have compact Clifford-Klein forms (Algorithm \ref{alg2}).}\label{tab:split}\\
\hline $\mathfrak{g}$ & $\mathfrak{h}$ \\\hline
$\mathfrak{e}_{6(6)}$ &  $ \mathfrak{sl}(2,\mathbb{R}) $ \\\hline $\mathfrak{e}_{6(6)}$ & $ 
  \mathfrak{sl}(3,\mathbb{R}) $ \\\hline $\mathfrak{e}_{6(6)}$ & $ 
  \mathfrak{sl}(5,\mathbb{R}) $ \\\hline $\mathfrak{e}_{6(6)}$ & $ 
  \mathfrak{sl}(6,\mathbb{R}) $ \\\hline $\mathfrak{e}_{6(6)}$ & $ 
  \mathfrak{sl}(2,\mathbb{R}) \oplus\mathfrak{sl}(2,\mathbb{R}) $ \\\hline $\mathfrak{e}_{6(6)}$ & $ 
  \mathfrak{sl}(3,\mathbb{R}) \oplus\mathfrak{sl}(2,\mathbb{R}) $ \\\hline $\mathfrak{e}_{6(6)}$ & $ 
  \mathfrak{sl}(3,\mathbb{R}) \oplus\mathfrak{sl}(3,\mathbb{R}) $ \\\hline $\mathfrak{e}_{6(6)}$ & $ 
  \mathfrak{sl}(4,\mathbb{R}) \oplus\mathfrak{sl}(2,\mathbb{R}) $ \\\hline $\mathfrak{e}_{6(6)}$ & $ 
  \mathfrak{sl}(2,\mathbb{R}) \oplus\mathfrak{sl}(2,\mathbb{R}) \oplus\mathfrak{sl}(2,\mathbb{R}) $ \\\hline $\mathfrak{e}_{6(6)}$ & $ 
  \mathfrak{sl}(3,\mathbb{R}) \oplus\mathfrak{sl}(3,\mathbb{R}) \oplus\mathfrak{sl}(2,\mathbb{R}) $ \\\hline $\mathfrak{e}_{6(6)}$ & $ 
  \mathfrak{sl}(4,\mathbb{R}) \oplus\mathfrak{sl}(2,\mathbb{R}) \oplus\mathfrak{sl}(2,\mathbb{R}) $ \\\hline
$\mathfrak{e}_{7(7)}$ &    $ \mathfrak{sl}(2,\mathbb{R}) $ \\\hline $\mathfrak{e}_{7(7)}$ & $ 
  \mathfrak{sl}(3,\mathbb{R}) $ \\\hline $\mathfrak{e}_{7(7)}$ & $ 
  \mathfrak{sl}(4,\mathbb{R}) $ \\\hline $\mathfrak{e}_{7(7)}$ & $ 
  \mathfrak{sl}(7,\mathbb{R}) $ \\\hline $\mathfrak{e}_{7(7)}$ & $ 
  \mathfrak{so}(5,5) $ \\\hline $\mathfrak{e}_{7(7)}$ & $ 
  \mathfrak{sl}(2,\mathbb{R}) \oplus\mathfrak{sl}(2,\mathbb{R}) $ \\\hline $\mathfrak{e}_{7(7)}$ & $ 
  \mathfrak{sl}(3,\mathbb{R}) \oplus\mathfrak{sl}(2,\mathbb{R}) $ \\\hline $\mathfrak{e}_{7(7)}$ & $ 
  \mathfrak{sl}(4,\mathbb{R}) \oplus\mathfrak{sl}(2,\mathbb{R}) $ \\\hline $\mathfrak{e}_{7(7)}$ & $ 
  \mathfrak{sl}(5,\mathbb{R}) \oplus\mathfrak{sl}(3,\mathbb{R}) $ \\\hline $\mathfrak{e}_{7(7)}$ & $ 
  \mathfrak{so}(5,5) \oplus\mathfrak{sl}(2,\mathbb{R}) $ \\\hline $\mathfrak{e}_{7(7)}$ & $ 
  \mathfrak{sl}(3,\mathbb{R}) \oplus\mathfrak{sl}(2,\mathbb{R}) \oplus\mathfrak{sl}(2,\mathbb{R}) $ \\\hline $\mathfrak{e}_{7(7)}$ & $ 
  \mathfrak{sl}(3,\mathbb{R}) \oplus\mathfrak{sl}(3,\mathbb{R}) \oplus\mathfrak{sl}(3,\mathbb{R}) $ \\\hline $\mathfrak{e}_{7(7)}$ & $ 
  \mathfrak{sl}(4,\mathbb{R}) \oplus\mathfrak{sl}(2,\mathbb{R}) \oplus\mathfrak{sl}(2,\mathbb{R}) $ \\\hline $\mathfrak{e}_{7(7)}$ & $ 
  \mathfrak{sl}(3,\mathbb{R}) \oplus\mathfrak{sl}(2,\mathbb{R}) \oplus\mathfrak{sl}(2,\mathbb{R}) \oplus\mathfrak{sl}(2,\mathbb{R}) $ \\\hline $\mathfrak{e}_{7(7)}$ & $ 
  \mathfrak{sl}(4,\mathbb{R}) \oplus\mathfrak{sl}(2,\mathbb{R}) \oplus\mathfrak{sl}(2,\mathbb{R}) \oplus\mathfrak{sl}(2,\mathbb{R}) $ \\\hline 
$\mathfrak{e}_{8(8)}$ &   $ \mathfrak{sl}(2,\mathbb{R}) $ \\\hline $\mathfrak{e}_{8(8)}$ & $ 
  \mathfrak{sl}(3,\mathbb{R}) $ \\\hline $\mathfrak{e}_{8(8)}$ & $ 
  \mathfrak{sl}(4,\mathbb{R}) $ \\\hline $\mathfrak{e}_{8(8)}$ & $ 
  \mathfrak{sl}(5,\mathbb{R}) $ \\\hline $\mathfrak{e}_{8(8)}$ & $ 
  \mathfrak{sl}(7,\mathbb{R}) $ \\\hline $\mathfrak{e}_{8(8)}$ & $ 
  \mathfrak{sl}(8,\mathbb{R}) $ \\\hline $\mathfrak{e}_{8(8)}$ & $ 
  \mathfrak{so}(5,5) $ \\\hline $\mathfrak{e}_{8(8)}$ & $ 
  \mathfrak{so}(7,7) $ \\\hline $\mathfrak{e}_{8(8)}$ & $ 
  \mathfrak{e}_{6(6)} $ \\\hline $\mathfrak{e}_{8(8)}$ & $ 
  \mathfrak{e}_{7(7)} $ \\\hline $\mathfrak{e}_{8(8)}$ & $ 
  \mathfrak{sl}(2,\mathbb{R}) \oplus\mathfrak{sl}(2,\mathbb{R}) $ \\\hline $\mathfrak{e}_{8(8)}$ & $ 
  \mathfrak{sl}(3,\mathbb{R}) \oplus\mathfrak{sl}(2,\mathbb{R}) $ \\\hline $\mathfrak{e}_{8(8)}$ & $ 
  \mathfrak{sl}(3,\mathbb{R}) \oplus\mathfrak{sl}(3,\mathbb{R}) $ \\\hline $\mathfrak{e}_{8(8)}$ & $ 
  \mathfrak{sl}(4,\mathbb{R}) \oplus\mathfrak{sl}(2,\mathbb{R}) $ \\\hline $\mathfrak{e}_{8(8)}$ & $ 
  \mathfrak{sl}(4,\mathbb{R}) \oplus\mathfrak{sl}(3,\mathbb{R}) $ \\\hline $\mathfrak{e}_{8(8)}$ & $ 
  \mathfrak{sl}(4,\mathbb{R}) \oplus\mathfrak{sl}(4,\mathbb{R}) $ \\\hline $\mathfrak{e}_{8(8)}$ & $ 
  \mathfrak{sl}(5,\mathbb{R}) \oplus\mathfrak{sl}(3,\mathbb{R}) $ \\\hline $\mathfrak{e}_{8(8)}$ & $ 
  \mathfrak{sl}(5,\mathbb{R}) \oplus\mathfrak{sl}(4,\mathbb{R}) $ \\\hline $\mathfrak{e}_{8(8)}$ & $ 
  \mathfrak{sl}(6,\mathbb{R}) \oplus\mathfrak{sl}(2,\mathbb{R}) $ \\\hline $\mathfrak{e}_{8(8)}$ & $ 
  \mathfrak{sl}(6,\mathbb{R}) \oplus\mathfrak{sl}(3,\mathbb{R}) $ \\\hline $\mathfrak{e}_{8(8)}$ & $ 
  \mathfrak{sl}(7,\mathbb{R}) \oplus\mathfrak{sl}(2,\mathbb{R}) $ \\\hline $\mathfrak{e}_{8(8)}$ & $ 
  \mathfrak{so}(4,4) \oplus\mathfrak{sl}(2,\mathbb{R}) $ \\\hline $\mathfrak{e}_{8(8)}$ & $ 
  \mathfrak{so}(4,4) \oplus\mathfrak{sl}(3,\mathbb{R}) $ \\\hline $\mathfrak{e}_{8(8)}$ & $ 
  \mathfrak{so}(4,4) \oplus\mathfrak{sl}(4,\mathbb{R}) $ \\\hline $\mathfrak{e}_{8(8)}$ & $ 
  \mathfrak{so}(5,5) \oplus\mathfrak{sl}(2,\mathbb{R}) $ \\\hline $\mathfrak{e}_{8(8)}$ & $ 
  \mathfrak{so}(5,5) \oplus\mathfrak{sl}(3,\mathbb{R}) $ \\\hline $\mathfrak{e}_{8(8)}$ & $ 
  \mathfrak{so}(6,6) \oplus\mathfrak{sl}(2,\mathbb{R}) $ \\\hline $\mathfrak{e}_{8(8)}$ & $ 
  \mathfrak{sl}(2,\mathbb{R}) \oplus\mathfrak{sl}(2,\mathbb{R}) \oplus\mathfrak{sl}(2,\mathbb{R}) $ \\\hline $\mathfrak{e}_{8(8)}$ & $ 
   \mathfrak{sl}(3,\mathbb{R}) \oplus\mathfrak{sl}(2,\mathbb{R}) \oplus\mathfrak{sl}(2,\mathbb{R}) $ \\\hline $\mathfrak{e}_{8(8)}$ & $ 
   \mathfrak{sl}(3,\mathbb{R}) \oplus\mathfrak{sl}(3,\mathbb{R}) \oplus\mathfrak{sl}(3,\mathbb{R}) $ \\\hline $\mathfrak{e}_{8(8)}$ & $ 
   \mathfrak{sl}(4,\mathbb{R}) \oplus\mathfrak{sl}(2,\mathbb{R}) \oplus\mathfrak{sl}(2,\mathbb{R}) $ \\\hline $\mathfrak{e}_{8(8)}$ & $ 
   \mathfrak{sl}(5,\mathbb{R}) \oplus\mathfrak{sl}(2,\mathbb{R}) \oplus\mathfrak{sl}(2,\mathbb{R}) $ \\\hline $\mathfrak{e}_{8(8)}$ & $ 
   \mathfrak{sl}(5,\mathbb{R}) \oplus\mathfrak{sl}(3,\mathbb{R}) \oplus\mathfrak{sl}(2,\mathbb{R}) $ \\\hline $\mathfrak{e}_{8(8)}$ & $ 
   \mathfrak{so}(5,5) \oplus\mathfrak{sl}(2,\mathbb{R}) \oplus\mathfrak{sl}(2,\mathbb{R}) $ \\\hline $\mathfrak{e}_{8(8)}$ & $ 
   \mathfrak{sl}(3,\mathbb{R}) \oplus\mathfrak{sl}(2,\mathbb{R}) \oplus\mathfrak{sl}(2,\mathbb{R}) \oplus\mathfrak{sl}(2,\mathbb{R}) $ \\\hline $\mathfrak{e}_{8(8)}$ & $ 
   \mathfrak{sl}(3,\mathbb{R}) \oplus\mathfrak{sl}(3,\mathbb{R}) \oplus\mathfrak{sl}(2,\mathbb{R}) \oplus\mathfrak{sl}(2,\mathbb{R}) $ \\\hline $\mathfrak{e}_{8(8)}$ & $ 
   \mathfrak{sl}(3,\mathbb{R}) \oplus\mathfrak{sl}(3,\mathbb{R}) \oplus\mathfrak{sl}(3,\mathbb{R}) \oplus\mathfrak{sl}(2,\mathbb{R}) $ \\\hline $\mathfrak{e}_{8(8)}$ & $ 
   \mathfrak{sl}(4,\mathbb{R}) \oplus\mathfrak{sl}(2,\mathbb{R}) \oplus\mathfrak{sl}(2,\mathbb{R}) \oplus\mathfrak{sl}(2,\mathbb{R}) $ \\\hline $\mathfrak{e}_{8(8)}$ & $ 
   \mathfrak{sl}(4,\mathbb{R}) \oplus\mathfrak{sl}(3,\mathbb{R}) \oplus\mathfrak{sl}(2,\mathbb{R}) \oplus\mathfrak{sl}(2,\mathbb{R}) $ \\\hline $\mathfrak{e}_{8(8)}$ & $ 
   \mathfrak{so}(4,4) \oplus\mathfrak{sl}(2,\mathbb{R}) \oplus\mathfrak{sl}(2,\mathbb{R}) \oplus\mathfrak{sl}(2,\mathbb{R}) $ \\\hline $\mathfrak{e}_{8(8)}$ & $ 
   \mathfrak{sl}(2,\mathbb{R}) \oplus\mathfrak{sl}(2,\mathbb{R}) \oplus\mathfrak{sl}(2,\mathbb{R}) \oplus\mathfrak{sl}(2,\mathbb{R}) \oplus\mathfrak{sl}(2,\mathbb{R}) $ \\\hline $\mathfrak{e}_{8(8)}$ & $ 
   \mathfrak{sl}(3,\mathbb{R}) \oplus\mathfrak{sl}(2,\mathbb{R}) \oplus\mathfrak{sl}(2,\mathbb{R}) \oplus\mathfrak{sl}(2,\mathbb{R}) \oplus\mathfrak{sl}(2,\mathbb{R}) $ \\\hline $\mathfrak{e}_{8(8)}$ & $ 
   \mathfrak{sl}(4,\mathbb{R}) \oplus\mathfrak{sl}(2,\mathbb{R}) \oplus\mathfrak{sl}(2,\mathbb{R}) \oplus\mathfrak{sl}(2,\mathbb{R}) \oplus\mathfrak{sl}(2,\mathbb{R}) $ \\\hline $\mathfrak{e}_{8(8)}$ & $ 
   \mathfrak{sl}(2,\mathbb{R}) \oplus\mathfrak{sl}(2,\mathbb{R}) \oplus\mathfrak{sl}(2,\mathbb{R}) \oplus\mathfrak{sl}(2,\mathbb{R}) \oplus\mathfrak{sl}(2,\mathbb{R})\oplus\mathfrak{sl}(2,\mathbb{R})\oplus\mathfrak{sl}(2,\mathbb{R}) $ \\\hline 
       $\mathfrak{f}_{4(4)}$ &   $ \mathfrak{sl}(2,\mathbb{R}) $ \\\hline $\mathfrak{f}_{4(4)}$ & $ 
   \mathfrak{sl}(3,\mathbb{R}) $ \\\hline $\mathfrak{f}_{4(4)}$ & $ 
   \mathfrak{sl}(4,\mathbb{R}) $ \\\hline $\mathfrak{f}_{4(4)}$ & $ 
   \mathfrak{so}(2,3) $ \\\hline $\mathfrak{f}_{4(4)}$ & $ 
   \mathfrak{sl}(3,\mathbb{R}) \oplus\mathfrak{sl}(2,\mathbb{R}) $ \\\hline $\mathfrak{f}_{4(4)}$ & $ 
   \mathfrak{sl}(2,\mathbb{R}) \oplus\mathfrak{sl}(2,\mathbb{R}) \oplus\mathfrak{sl}(2,\mathbb{R}) $ \\\hline 
 $\mathfrak{g}_{2(2)}$ & $ 
   \mathfrak{sl}(2,\mathbb{R}) $ \\\hline
\end{longtable}

\section{Proof of Theorem \ref{thm:sl2}}
\begin{proof}
Note that  $i_{\ast}([H_{U}/K_{H}])$ is non-zero in $H_{\ast}(G_{U}/K,\mathbb{R})$, by Theorems \ref{thm:omega} and \ref{thm:vanishing-omega}. By Poincar\'e duality there exists $a\in H^{\dim H_{u}/K_{H}}(G_{U}/K,\mathbb{R})$ such that $i^{\ast}(a)\neq 0.$ One can see that $H^{\dim H_{u}/K_{H}}(H_{U}/K_{H},\mathbb{R})$ is one-dimensional. It follows that
the map 
$$i^{\ast}:H^{\dim H_{U}/K_{H}}(G_{U}/K, \mathbb{R})\rightarrow H^{\dim H_{u}/K_{H}}(H_{U}/K_{H},\mathbb{R}) \eqno (1)$$ 
is surjective. Use the following observation from \cite{BC}, Theorem 1.1.
\vskip6pt
\noindent {\sl Let $A$ be a semisimple connected compact Lie group and $U$ a closed and connected subgroup. Then $H^{2}(A/U, \mathbb{R})$ is canonically isomorphic to the center of $\mathfrak{u}^{\ast}.$}
\vskip6pt
Assume that $G/H$ admits a compact Clifford-Klein form. Note that $\dim\,H_{U}/K_{H}$=2 so by the surjectivity of $(1)$ we have
$$H^{2}(G_{U}/K,\mathbb{R})\neq 0.$$
But Theorem 1.1 in \cite{BC} (which we  have mentioned above) implies
$$H^{2}(G_{U}/K,\mathbb{R})= 0,$$
a contradiction.
\end{proof}


\begin{thebibliography}{ABCDE}


\bibitem[BEN]{BEN} Y. Benoist, {\it Actions propres sur les espaces homog\`enes r\'eductifs}, Ann. of Math. 144 (1996), 315-347.

\bibitem[BL]{BL} Y. Benoist, F. Labourie {\it  Sur les espaces homog\`enes mod\`eles de vari\'et\'es compactes}, Publications Math\'ematiques de l'I.H.\'E.S. 76 (1992), 99-109. 


\bibitem[BC]{BC} I. Biswas, P. Chatterjee {\it Second cohomology of compact homogeneous spaces}, Int. J. Math. 24 (2013).

\bibitem[BJSTW]{BJSTW} M. Boche\'nski, P. Jastrz\c ebski, A. Szczepkowska, A. Tralle, A. Woike, {\it NoCK - Computing obstruction for compact Clifford-Klein form. A GAP package}, available at \url{https://pjastr.github.io/NoCK/}.

\bibitem[BT]{BT} M. Boche\'nski, A. Tralle, {\it On solvable compact Clifford-Klein forms}, Proc. Amer. Math. Soc. 145(2017), 1819-1832

\bibitem[BT1]{BT1} M. Boche\'nski, A. Tralle, {\it On Hirzebruch's proportionality principle and the non-existence of amenable Clifford-Klein forms of certain homogeneous spaces}, arXiv:1705.03221.

\bibitem[CM62]{CM62} E. Calabi, L. Markus, {\it Relativistic space forms}, Ann. Math. 75(1962), 63-76.

\bibitem[C]{C} D. Constantine, {\it Compact Clifford-Klein forms - geometry, topology and dynamics}, Geometry, topology and dynamics in negative curvature, 110-145, London Math. Soc. Lecture Note Ser. 425, Cambridge Univ. Press, 2016.

\bibitem[DFG]{DFG} H. Dietrich, P. Faccin, W. A. de Graaf, {\it Regular subalgebras and nilpotent orbits of real graded Lie algebras}, Journal of Algebra, 423 (2015), 1044--1079. 

\bibitem[DFG1]{DFG1} H. Dietrich, P. Faccin, W. A. de Graaf, {\it Corelg: computing with real Lie groups, A GAP 4 package}, available at \url{http://users.monash.edu/~heikod/corelg/}.

\bibitem[D]{D} E. B. Dynkin, {\it Semisimple subalgebras of semisimple Lie algebras}, Mat. Sbornik N.S., 30(72):349–462, 1952. English translation and commentary: E. B. Dynkin, G. M. Seitz, A. L. Onishchik,  {\it Selected Papers of E.B. Dynkin with Commentary}, American Mathematical Soc., 2000.

\bibitem[FHT]{FHT} Y. F\'elix, S. Halperin, J.-C. Thomas, {\it Rational homotopy theory}, Springer, Berlin, 2001.

\bibitem[GAP]{GAP} {\it The GAP Group, GAP - groups, algorithms and programming}, v.4.8, available at \url{https://www.gap-system.org/}.

\bibitem[G]{G} W. A. de Graaf, {\it  SLA - computing with Simple Lie Algebras. A GAP package}, available at \url{http://www.science.unitn.it/~degraaf/sla.html}.

\bibitem[G1]{G1} W. A. de Graaf, {\it Constructing semisimple subalgebras of simpe Lie algebras}, J. Algebra 325(2011), 416-430.

\bibitem[Ka]{Ka} F. Kassel, {\it Deformation of proper actions on reductive homogeneous spaces}, Math. Ann. 353(2012), 599-632.

\bibitem[Kob89]{Kob89} T. Kobayashi, {\it Proper actions on a homogeneous space of reductive type}, Math. Ann. 285(1989), 249-263.

\bibitem[Kob92]{Kob92} T. Kobayashi, {\it A necessary condition for the existence of compact Clifford-Klein forms of homogeneous spaces of reductive type}, Duke Math. J. 67(1992), 653-664.

\bibitem[Kob93]{Kob93} T. Kobayashi, {\it On discontinuous groups acting on homogeneous spaces with noncompact isotropy subgroups}, J. Geom. Phys. 12(1993), 133-144.

\bibitem[Kob96]{Kob96} T. Kobayashi, {\it Discontinuous groups and Clifford-Klein forms of pseudo-Rimemannian homogeneous manifolds}, Prespectives in math. 17(1996), 99-164.

\bibitem[Kob98]{Kob98} T. Kobayshi, {\it Deformation of compact Clifford-Klein forms of indefinite-Riemannian homogeneous manifolds}, Math. Ann. 310(1998), 394-408.

\bibitem[KO]{KO} T. Kobayashi, K. Ono, {\it Note on Hirzebruch's proportionality principle}, J. Fac. Sci. Univ. Tokyo 37 (1990), 71-87.

\bibitem[KY]{KY} T. Kobayashi, T. Yoshino, {\it Compact Clifford-Klein forms of symmetric spaces - revisited}, Pure Appl. Math. Quart. 1(2005), 591-663


\bibitem[LZ]{LZ} F. Labourie, R. J. Zimmer {\it On the non-existence of cocompact lattices for $SL(n)/SL(m)$}, Math. Res. Letters 2  (1995), 75-77.

\bibitem[LMZ]{LMZ} F. Labourie, S. Mozes, R. Zimmer, {\it On manifolds locally modelled on non-Riemannian homogeneous spaces}, Geom. Funct. Anal. 5(1995), 955-965.

\bibitem[MA]{MA} G. Margulis, {\it Existence of compact quotients of homogeneous spaces, measurably proper actions, and decay of matrix coefficients}, Bull. Soc. Math. France 125 (1997), 447-456.

\bibitem[M]{M} Y. Morita {\it A topological necessary condition for the existence of compact Clifford-Klein forms}, J. Differential Geom. 100 (2015), 533-545.

\bibitem[OH]{OH} H. Oh, {\it Tempered subgroups and representations with minimal decay of matrix coefficients}, Bull. Soc. Math. France 126 (1998), 355-380.


\bibitem[O]{O} A.L. Onishchik, {\it Topology of compact transformation groups}, Leipzig, 1993.

\bibitem[O1]{O1} A. L. Onishchik, {\it Lectures on Real Semisimple Lie Algebras and their Representations}, European Math. Soc., Z\''urich, 2004 

\bibitem[OW]{OW} H. Oh, D. Witte-Morris, {\it Compact Clifford-Klein forms of homogeneous spaces of $SO(2,n)$,} Geom. Dedicata 89(2002), 25-57.

\bibitem[SH]{SH} Y. Shalom, {\it Rigidity, unitary representations of semisimple groups, and fundamental groups of manifolds with rank one transformation group}, Ann. of Math. 152 (2000), 113-182.
 
\bibitem[Th]{Th} N. Tholozan,{\it Volume and non-existence of compact Clifford-Klein forms}, arXiv: 1511.09448[math.GT].

\bibitem[T]{T} S. Terzi\v c, {\it Cohomology with real coefficients of generalized symmetric spaces}, Fundam. Prikl. Mat. 7(2001), 131-157 (in Russian).

\bibitem[TO]{TO} A. Tralle, J. Oprea, {\it Symplectic manifolds with no Kaehler structure}, Spinger, Berlin, 1997.



\end{thebibliography}
\end{document}